\documentclass[12pt,a4paper]{amsart}
\usepackage[utf8]{inputenc}
\usepackage{amsmath}
\usepackage{amsfonts}
\usepackage{amssymb}
\usepackage{mathtools}
\usepackage{amsthm}
\usepackage{graphicx}
\usepackage{hyperref}
\usepackage{tikz-cd}

\author{Robert Śmiech}
\title{Singular contact varieties}

\address{Robert \'Smiech, University of Warsaw, Faculty of Mathematics, Informatics and Mechanics, ul. Banacha 2, 02-097 Warszawa and IDEAS NCBR, ul. Chmielna 69, 00-801 Warszawa \\
ORCiD: 0000-0002-5654-8433}
\email{r.smiech@mimuw.edu.pl}

\newtheorem{thm}{Theorem}[section]
\newtheorem{lem}[thm]{Lemma}
\newtheorem{prop}[thm]{Proposition}
\newtheorem{cor}[thm]{Corollary}

\theoremstyle{definition}
\newtheorem{rmk}[thm]{Remark}
\newtheorem{defn}[thm]{Definition}

\newtheorem{exa}[thm]{Example}

\newtheorem{set}[thm]{Setting}

\DeclareMathOperator{\Pic}{Pic}

\DeclareMathOperator{\Exc}{Exc}

\DeclareMathOperator{\SL}{SL}
\DeclareMathOperator{\End}{End}
\DeclareMathOperator{\Hom}{Hom}

\DeclareMathOperator{\id}{id}
\DeclareMathOperator{\codim}{codim}

\DeclareMathOperator{\cone}{Cone}

\DeclareMathOperator{\diag}{diag}
\DeclareMathOperator{\stab}{Stab}

\newcommand{\Z}{\mathbb{Z}}
\newcommand{\R}{\mathbb{R}}
\newcommand{\C}{\mathbb{C}}
\newcommand{\p}{\mathbb{P}}

\begin{document}
	
	\begin{abstract}
		In this note we propose the generalization of the notion of a holomorphic contact structure on a manifold (smooth variety) to varieties with rational singularities and prove basic properties of such objects. Natural examples of \textit{singular contact varieties} come from the theory of nilpotent orbits: every projectivization of the closure of a nilpotent orbit in a semisimple Lie algebra satisfies our definition after normalization. We show the correspondence between symplectic varieties with the structure of a $\C^*$-bundle and the contact ones along with the existence of stratification \textit{\`a la} Kaledin. In the projective case we demonstrate the equivalence between crepant and contact resolutions of singularities, show the uniruledness and give a full classification of projective contact varieties in dimension 3.
	\end{abstract}
	\maketitle
	\section{Introduction}
	\subsection{Background}
	\textit{Contact geometry is an odd-dimensional counterpart of symplectic geometry} is a slogan that is frequently employed when discussing contact structures on manifolds. In the complex algebraic setting, smooth projective varieties equipped with a holomorphic contact structure are researched for their own merit: they are uniruled, so to study them one can employ methods using rational curves and their contractions \cite{KPSW}. Nevertheless, their classification is still open in the prime Fano case and conjecturally (\cite{LeBS94}) the only smooth Fano contact varieties are projectivizations of minimal nilpotent orbits for simple Lie groups. \par
	In the meantime, the seminal paper of Beauville (\cite{Bea00}) marked the birth of the singular symplectic geometry, that has since then become an important field of research, see for example the classical, although not very recent survey by Fu \cite{Fu05}. Fundamental examples of (singular) symplectic varieties come from the Lie theory: every coadjoint orbit of a complex semisimple Lie group is equipped with the Kostant-Kirillov symplectic form, and normalizations of the closures of nilpotent orbits have symplectic singularities. Some properties of general symplectic varieties, like the existence of the canonical stratification can be easily seen on those model examples.\par
	On the other hand, the idea of studying some generalization of the contact structure in the singular setting did not attract a comparable attention. The author is only aware of three works exploring such ideas. Namely, Campana and Flenner proposed a definition of a contact singularity \textit{\`a la} Beauville \cite{CF01} and Fu \cite{Fu06} studied the contact resolutions of projectivized nilpotent orbits. Therefore the purpose of this note it to propose the notion of a \textit{singular contact variety} that allows extending the classical construction of symplectization for smooth contact varieties and relate it to objects mentioned above. Some of the results that we discuss have analogues in the symplectic world, but most of our proofs do not rely on the symplectization correspondence. Additionally, in the projective case singular contact varieties are uniruled, so we can endow our toolbox with methods utilizing the existence of rational curves.
	\subsection{Overview}
	As we have just indicated, we begin by giving our definition of a singular contact variety (Definition \ref{defn.contact}) and discussing its relation with the notion of \textit{a contact singularity} of Campana and Flenner \cite{CF01}. Our motivation is twofold: first, we want to generalize the notion of a holomorphic contact manifold in such a way that we obtain some new examples but the situation is not entirely different from the smooth case. Moreover, with our definition we can extend the classical construction of the symplectization, i.e.~we present a correspondence between singular contact varieties and principal $\C^*$-bundles equipped with a homogeneous symplectic form (Theorem~\ref{thm.symplectization}). This correspondence allows us to prove an analogue of canonical stratification that for symplectic varieties was obtained by Kaledin~\cite{Kal06} (Theorem~\ref{thm.contact.stratification}). We check that projectivized normalization of nilpotent orbit closures of semisimple Lie algebras satisfy our definition (Proposition~\ref{prop.orbits}). Finally, we study when the contact structure descends to a finite quotient (Theorem~\ref{thm.quotient}).\par 
	Then we restrict ourselves to the projective category. In this setting we can show, using a recent result of Cao and H\"oring~\cite{CH22} that as in the smooth case, projective contact varieties are uniruled (Proposition~\ref{prop.not.psef}). This is a crucial property that opens up possibilities for utilizing Mori-theoretic methods in the future work. We also study the relation between being crepant and preserving contact structure for a birational morphism (Theorem~\ref{thm.contact.is.crepant} and Theorem~\ref{thm.crepant.is.contact}). \par 
	To conclude our work, we first study two examples, that we first construct using the finite quotient and our Theorem~\ref{thm.quotient} (Example~\ref{exa.fav} and Example~\ref{exa.p5}), but then we show that they are in fact isomorphic to union of some projectivized nilpotent orbits (Remark~\ref{rmk.example}). Then we use the developed tools to give a full classification of singular projective contact varieties in dimension 3 (Theorem~\ref{thm.threefolds} and Theorem~\ref{thm.3fold.fano}). As a by-product of our theorems, we give an example of a projective threefold whose singularities are nonrational, but at the same time they satisfy the definition of Campana-Flenner (Example~\ref{exa.CP}).
	\subsection{Notation and conventions}
	We work over the field of complex numbers, so in particular by symplectic (contact) structure we always mean a holomorphic one. Since we are interested in the algebraic category, a \textit{manifold} for us is a smooth algebraic variety. For a line bundle $L$ we denote by $L^\bullet$ the total space of its dual with the image of the zero section removed, it is clearly a $\mathbb{C}^*$--principal bundle. Moreover we do not make a distinction between line bundles and Cartier divisors. For the projectivized bundles we follow the Grothendieck convention, i.e.~$\p(\mathcal{E})$ is the space parametrizing one dimensional quotients of $\mathcal{E}$. By \textit{a curve} we mean an effective, irreducible and reduced 1-cycle.
	\subsection{Acknowledgements}
	The author would like to thank his advisor, Jaros{\l}aw Buczy{\'n}ski, for his guidance and patience during the preparation of this project that is a part of author's PhD dissertation. Moreover, he is grateful to Maria Donten-Bury, Jaros{\l}aw Wi{\'s}niewski and Andreas H{\"o}ring for enlightening discussions. The author would also like to express his gratitude to the anonymous referee, whose careful reading and helpful suggestions greatly improved this manuscript. Finally, during his PhD studies, the author was a member of the \textit{Kartezjusz} programme and he was supported by projects "Singular contact varieties" 2021/41/N/ST1/03812 and "Complex contact manifolds and geometry of secants" 2017/26/E/ST1/00231 of the National Science Centre of Poland.
	\section{Definition and basic properties}
	\begin{defn}\label{defn.contact}
		A contact variety is an algebraic variety $X$ over $\mathbb{C}$ of odd dimension $2n+1$ ($n \ge 0$) with rational singularities and a globally defined line bundle $L$ such that on the smooth locus $X_{sm}$ we have an exact sequence of vector bundles:
		\begin{displaymath}
			0 \rightarrow F \rightarrow TX_{sm} \xrightarrow{\vartheta} L_{|X_{sm}} \rightarrow 0
		\end{displaymath}
		which defines contact structure on $X_{sm}$, i.e. $d\vartheta : \bigwedge^2 F \rightarrow  L_{|X_{sm}}$ is nowhere degenerate. Equivalently one can demand that  $\vartheta \wedge (d \vartheta)^{\wedge n}$ as an element of $H^0(X_{sm}, \Omega^{2n+1}_{|X_{sm}} \otimes L^{n+1}_{|X_{sm}})$ has no zeroes. We will sometimes call $L$ from the definition \textit{contact line bundle}. Depending on which elements of the structure we need to emphasize, we will write $(X,L)$, $(X,L, \vartheta)$, $(X,F,L)$, etc.
	\end{defn}
	\begin{rmk}
		Note that in our definition we allow contact varieties of dimension 1. At the same time, we require normality, so they are necessarily smooth. Clearly, any smooth curve $C$ is equipped with a trivial contact structure by putting $L = TC$, $\vartheta = \id$. We explicitly include the one-dimensional case to easily state the Stratification Theorem~\ref{thm.contact.stratification}.
	\end{rmk}
	\begin{rmk}
		It is clear that singularities of such varieties are contact in the sense of Campana and Flenner \cite{CF01}. However, in their definition they do not assume the rationality of a singularity, but only normality. In this way, the notion of a contact singularity proposed by Campana and Flenner differs not only from our notion of a contact variety, but also from Beauville's definition of a symplectic singularity \cite{Bea00a}. Although Beauville does not assume rationality explicitly, the condition on the existence of extension of the symplectic 2-form to a resolution is equivalent to it, as shown by \cite[Thm 6]{Nam01}. We will construct an instance of a projective threefold with such singularities (nonrational but equipped with a contact form on the smooth locus) in Example~\ref{exa.CP}.
	\end{rmk}
	 We will note some easy consequences of the given definition, which will be useful later:
	\begin{prop}\label{prop.cartier}
		$-K_X$ is a Cartier divisor and we have $\mathcal{O}(-K_X) = L^{\otimes (n+1)}$ in $\Pic(X)$. Consequently, singularities of $X$ are Gorenstein and canonical.
	\end{prop}
	\begin{proof}
		By an abuse of notation we will use the same symbol $L$ for the contact line bundle and the corresponding Weil divisor. On $X_{sm}$ we have an equality of classes of Weil divisors $(n+1)L_{|X_{sm}} = -K_{X_{sm}}$ given by $\vartheta \wedge (d \vartheta)^{\wedge n}$. By the normality of $X$ we can prolong this equality to whole $X$ by taking unique closures of both divisors. Then, because $L$ corresponds to a line bundle, so does the canonical class. Moreover, singularities of $X$ are rational by definition, so they are in particular Cohen-Macaulay and Gorenstein. Finally, the pair $(X,0)$ is canonical by \cite[Cor. 5.24]{KM92}.
	\end{proof}
	There is a classical construction of the symplectization, described for example in \cite[Thm. E.6]{Bu09} that associates to each contact manifold the total space of a principal $\C^*$-bundle equipped with a homogeneous symplectic form. We have defined contact varieties in such a way to preserve this correspondence in the singular setting. To be precise, we can show the following:
	\begin{thm}\label{thm.symplectization}
		Let $X$ be a contact variety with the contact line bundle $L$ and a twisted form $\vartheta$. Then $L^\bullet$ is a variety with symplectic singularities such that the symplectic form $\omega$, induced by $\vartheta$ is homogeneous of weight 1 with respect to the natural $\C^*$-action on $L^\bullet$. \\
		On the other hand, if $\pi: Y \rightarrow Z$ is a principal $\C^*$-bundle having the structure of a symplectic variety such that the 2-form is homogeneous of weight 1, then $Z$ is a contact variety with the twisted form induced from the symplectic one.
	\end{thm}
	\begin{proof}
		First recall that by \cite[Theorem 6]{Nam01} a normal variety is symplectic if and only if it has rational Gorenstein singularities and its smooth part admits a holomorphic symplectic form. For the first statement, observe that singularities of $L^\bullet$ are rational and Gorenstein, as it is a locally trivial $\C^*$-bundle over $X$. Therefore it is enough to provide a homogeneous holomorphic symplectic form on $L^\bullet_{sm}$. But this space is nothing else that the principal bundle over the $X_{sm}$, so one can do it using the standard construction \cite[Section C.5]{Bu09}. \\
		For the second claim, $Z$ has rational singularities as a base of a locally trivial fibration $Y$ with rational singularities  and the fiber $\C^*$. Moreover, there exists a globally defined line bundle on $Z$, obtained as the dual of the line bundle associated to $Y$. Analogously as before, $Z_{sm} = \pi (Y_{sm})$, so we construct a twisted form on $Z_{sm}$ using the construction for the smooth varieties \cite[Prop. C.16]{Bu09}.
	\end{proof}
	
	\subsection{Kaledin's stratification for contact varieties}
	The fundamental application of the correspondence presented above is to provide an analogue of Kaledin's stratification for symplectic varieties \cite{Kal06}. Namely, we have:
	\begin{thm}\label{thm.contact.stratification}
	Let $(X,L,\vartheta)$ be a contact variety. Then there exists a finite stratification $ X = X_0 \supset X_1 \supset ... \supset X_k$ such that:
	\begin{enumerate}
		\item $X_{i+1}$ is the singular part of $X_{i}$.
		\item The normalization of every irreducible component of each stratum is a contact variety.
	\end{enumerate}
\end{thm}
\begin{proof}
	Let $(L^\bullet, \omega)$ be the symplectization of $(X, L, \vartheta)$ and denote by $\pi$ the natural projection. Consider the symplectic stratification of $L^\bullet$ existing by \cite[Thm 2.3]{Kal06}. To prove our theorem it is enough to show that all the strata and their normalizations have a structure of a principal $\C^*$-bundle induced from the one on $L^\bullet$ and that the induced forms are homogeneous of weight $1$. \\
	First, observe that the singular locus of $L^\bullet$ must necessarily be preserved by the $\C^*$ action, so any of its components $Z$ still admits the structure of a principal $\C^*$-bundle. Then, the action can be lifted to the normalization $Z^n$. To see it, consider the diagram:
	\begin{center}
		\begin{tikzcd}
			\C^* \times Z^n \arrow[r, dashed] \arrow[d, "\id \times \eta"] & Z^n \arrow[d, "\eta"] \\
			\C^* \times Z \arrow[r] & Z,
		\end{tikzcd}
	\end{center}
	where $\eta$ is the normalization morphism and the lower horizontal arrow comes from the action. The existence of the dashed arrow comes from the universal property of normalization and it defines the $\C^*$ action on the normalization. Moreover, $Z^n$ also admits the structure of a principal $\C^*$-bundle, as it is equal to the pullback (via the normalization morphism of $\pi(Z)$) of the principal bundle on $\pi(Z)$. \\
	Now we need that the induced forms have the same weight as $\omega$. To that end, observe that in Kaledin's proof the symplectic forms on the strata agree with the globally defined Poisson structure. If we start with the symplectic form on $Y_{sm}$ which is homogeneous of weight $1$, then by \cite[Lemma D.15]{Bu09} the Poisson bracket is homogeneous of weight $-1$. Now, the bracket extends over the singular locus and it is still homogeneous of weight $-1$, so the induced forms on strata, which have to agree with the bracket have weight $1$.
\end{proof}
	\begin{cor}\label{cor.stratification}
		Codimensions of the strata are even. Moreover, if singularities of contact variety are terminal, then $\codim(X_{sing}) \ge 4$.
	\end{cor}
\begin{rmk}
	Campana and Flenner showed \cite[Thm 3.5]{CF01} that contact singularities in their sense cannot be isolated and posed a question whether there are contact singularities for which the singular stratum has even dimension. As we have just seen, this is not possible with the assumption of rationality, but the question remains open in their setting.
\end{rmk}
Just as in the case of symplectic varieties, the model examples come from the Lie theory:
\begin{prop}\label{prop.orbits}
	For a semisimple group $G$ with a Lie algebra $\mathfrak{g}$ consider a nilpotent orbit $\mathcal{O} \subset \mathfrak{g}^*$ and its projectivization $\p(\mathcal{O}) \subset \p(\mathfrak{g}^*)$. Then the normalization of its closure $\overline{\p(\mathcal{O})}^n$ is a contact variety. 
\end{prop}
\begin{proof}
	Recall that the normalization of a nilpotent orbit closure for a semisimple Lie group is a symplectic variety for a Kostant-Kirillov 2-form, as observed in \cite[2.5]{Bea00a}. Moreover, every nilpotent orbit is preserved by a natural $\C^*$-action on $\mathfrak{g}^* \setminus \{ 0\}$ and the symplectic form is homogeneous of weight 1. Therefore, as the normalization commutes with the projectivization, we are in the situation of Theorem \ref{thm.symplectization} and conclude that the quotient by the $\C^*$-action, i.e.~the projectivization admits a contact structure. Note that the contact line bundle comes from the restriction of $\mathcal{O}_{\p(\mathfrak{g}^*)}(1)$.
\end{proof}
As the nilpotent orbits for semisimple groups are partially ordered by inclusion in closure, they provide an excellent illustration of the stratification. Nevertheless, as we will see later, even in the projective case there are many other examples.
\subsection{Quotients by finite groups}
	As a last general result, we will prove a criterion that determines when a finite quotient of a contact variety has an induced contact structure. It will be useful for constructing examples. We let $G$ be a finite group acting on a contact variety (or manifold) $X$ that preserves the contact distribution. We start our considerations by taking a closer look at the contact line bundle $L$. In the smooth case, where the contact distribution is defined everywhere, this line bundle is automatically linearized, i.e.~there exists a canonical lift of the action of $G$ to the total space of $L$. Indeed, $L$ is defined as the quotient of $TX$ by $F$ and by the definition of $G$, both of those vector bundles come with a $G$-linearization. In the singular case we need a more subtle argument:
	\begin{prop}\label{prop.linearization}
		Let $(X,F,L)$ be a singular contact variety and $G$ a group of automorphisms preserving the contact distribution $F$. Then $L$ is equipped with the canonical $G$-linearization.
	\end{prop}
	\begin{proof}
		Consider the product $G \times X$ with two morphisms to $X$: $\alpha$ corresponding to the action and $\pi_2$ -- the projection from the product. Then by \cite[Lemma 2.9]{Bri15} the bundle $L$ is linearized if and only if $\pi_2^* L \simeq \alpha^*L$. On the smooth locus we have a canonical action of $G$ on $L$, as it is a quotient of $TX_{sm}$ by $F$ there. Moreover, $G$ acts on $TX_{sm}$ preserving $F$, so we have the above mentioned isomorphism of pullback bundles on $X_{sm}$. As the singular locus has codimension at least 2 by normality, this isomorphism extends to the whole $X$. Indeed, on any open $U \subset X$ trivializing both bundles such isomorphism is given by a nowhere vanishing regular function $f$ defined on $U \cap X_{sm}$. Any such $f$ can be extended to a nowhere vanishing regular function on the whole $U$ by normality of $X$. The resulting linearization is canonical, as the unique extension of the canonically defined linearization on the smooth locus.
	\end{proof}
	To formulate our statement we need to recall the definition of a descent of a vector bundle:
	\begin{defn}
	Let $Y$ be a $G$-variety with a good quotient $\pi \colon Y \rightarrow Y/G$. The pullback of any vector bundle on $Y/G$ has a natural structure of a $G$-vector bundle on $Y$. We say that a $G$-vector bundle $\mathcal{E}$ on $Y$ \textit{descends} if there exists a vector bundle $\mathcal{E}'$ on $Y/G$ such that we have a $G$-equivariant isomorphism $\mathcal{E} \simeq_G \pi^* \mathcal{E}'$. This vector bundle, if it exists, is unique and equal to $(\pi_* \mathcal{E})^G$. We will frequently call it \textit{the descent} of~$\mathcal{E}$.
\end{defn}
	The question on the existence of the descent has a precise answer, known in the literature as the Kempf lemma \cite[Thm 2.3]{DN89}. A $G$-vector bundle on a variety descends to the quotient if and only if the stabilizer of any closed point from a closed orbit acts trivially on the fiber. \par 
	 It will be useful to have estimates on the dimension of subschemes of points stabilized by some $g \in G$:
	\begin{lem}\label{lem.components}
	Let $(X,F,L)$ be a smooth contact manifold of dimension $2n +1 \ge 3$, $G$ a finite subgroup of automorphisms preserving the contact distribution $F$ and $Z^g$ the subscheme of points of $X$ stabilized by some $g \in G \setminus \{ \id \}$ (it can have multiple components). Then $Z^g$ does not have any divisorial component, i.e.~its codimension in $X$ is at least 2.
\end{lem}
\begin{proof}
	Let us pick a smooth point $x$ in some component $Y \subset Z^g$. From the Luna slice theorem \cite[Thm~5.4]{Dre04} we have $T_x Y \simeq (T_xX)^G$, so to calculate the dimension of $Y$ at each point it is enough to determine the dimension of the invariant subspace of the tangent space. To that end, define $H$ to be the subgroup of $G$ generated by $g$ and observe that $(T_xX)^G \subset (T_x X)^H$, so we will consider this bigger subspace. Denote by $\phi_g$ the automorphism of $T_x X$ induced by $g$. By the definition of $G$ we have $\phi_g (F) \subset F$ and the invertibility of $\phi_g$ implies that $\phi_g (F) = F$. The short exact sequence appearing in Definition~\ref{defn.contact} implies the existence of a direct sum decomposition $T_x X = L_x \oplus F_x$ that in our case is in the category of $G$-modules. $\phi_g$ is an invertible linear map such that $\phi_g^k = \id$ for some $k \in \Z_{>0}$, so by picking a diagonal base for $F$ we obtain $\phi_g = \diag(\xi_0, \xi_1,...,\xi_{2n})$, where $\xi_i$ are roots of unity of degree dividing $k$ and let $\xi_0$ be the root acting on $L_x$. The nondegeneracy of the contact form implies that there is an isomorphism $F \simeq F^* \otimes L$. This isomorphism gives us the following equality of sets: $\{ \xi_1, ...,\xi_{2n}\} = \{ \xi^{-1}_1 \cdot \xi_0, ..., \xi_{2n}^{-1} \cdot \xi_0 \}$. \par
	Now, if $\xi_0\ne 1$ then it follows that it is not possible for all other roots to be equal to 1, so ($T_xX)^H$ has codimension at least 2 in $T_xX$, so we conclude. If $\xi_0 = 1$, then the equality of sets above is not enough to conclude. Recall that we have $\mathcal{O}_X(-K_X) = L^{\otimes(n+1)}$, so in this case the action of $\phi_g$ on $\mathcal{O}_X(-K_X)_x$ is also trivial. But this implies that $\xi_0 \cdot \xi_1 \cdot ... \cdot \xi_{2n} =1$, so it is not possible that only one root differs from 1 and we conclude as before.
\end{proof}
From the proof above one can also deduce a statement on the singularities of the quotient:
\begin{cor}\label{cor.sing.quot}
	In the setting of Lemma~\ref{lem.components}, if $\pi \colon X \rightarrow X/G$ is the quotient map, then the variety $X/G$ is singular along $\pi(Z^g)$ for every $g \in G \setminus \{ \id \}$.
\end{cor}
\begin{proof}
	Recall the classical theorem attributed to Chevalley-Shephard-Todd \cite[Ch.~6]{Ben93}: the quotient is smooth if and only if the stabilizer of each point is generated by pseudoreflections, that is elements that fix pointwise a codimension 1 subvariety containing~$x$. But we have just seen in the proof of Lemma~\ref{lem.components} that in our case the codimension is at least 2. Therefore, the image of any point with a nontrivial stabilizer is singular.
\end{proof}
Now observe that if $X$ is a contact variety with a nonempty singular locus, then it is not possible for a singular point of $X$ to be mapped to a smooth one in the quotient (by a finite group preserving the contact distibution).\footnote{The author thanks the anonymous referee for this observation.} Indeed, by Lemma~\ref{lem.components} and the normality of $X$ the locus where the map $\pi \colon X \rightarrow X/G$ is not \'etale has codimension at least 2 in $X$ (such maps are sometimes called \textit{quasi-\'etale}). By the purity of the branch theorem \cite{Nag59}, the morphism $\pi$ is \'etale over any smooth point of $X/G$, so the preimage of $(X/G)_{sm}$ is smooth.

 Let us put $\mathcal{X} = X_{sm} \setminus \bigcup_{g \in G \setminus \{\id \}} Z^g$ and note that the discussion above shows that $\mathcal{X}$ is precisely the preimage of $(X/G)_{sm}$ and that the morphism $\pi$ is \'etale on $\mathcal{X}$. Moreover, every $G$-linearized vector bundle on $\mathcal{X}$ has a descent (to a vector bundle on $(X/G)_{sm}$), we have an isomorphism of tangent spaces $T\mathcal{X} \simeq \pi^* T(X/G)_{sm}$ \cite[Ch.~I, Prop.~2.9]{Mil13} so the tangent bundle of~$\mathcal{X}$ descends to the tangent bundle of $(X/G)_{sm}$. This discussion guides us to propose the following definition: 
\begin{defn}\label{defn.descent}
	Let $(X,F,L, \vartheta)$ be a contact variety of dimension $2n+1 \ge 3$ and $G$ a finite group of automorphisms preserving $F$. Let us take the contact exact sequence of $X$ restricted to $\mathcal{X}$:
	\begin{equation}\label{eqtn.restricted.contact}
		0 \rightarrow F_{|\mathcal{X}} \rightarrow T\mathcal{X} \xrightarrow{\vartheta_{|\mathcal{X}}} L_{|\mathcal{X}} \rightarrow 0.
	\end{equation}
	Consequently, we have a sequence of vector bundles on $(X/G)_{sm}$:
	\begin{equation}\label{eqtn.descent}
		0 \rightarrow (\pi_*F_{|\mathcal{X}})^G \rightarrow T(X/G)_{sm} \xrightarrow{\vartheta'} (\pi_*L_{|\mathcal{X}})^G, 
	\end{equation}
	that is exact, as it comes from the application of two left exact functors, $\pi_*$ and $(\cdot)^G$. We say that $X/G$ has \textit{the induced contact structure} if the above sequence of vector bundles on $(X/G)_{sm}$ gives a contact structure on $X/G$. Precisely, we demand the surjectivity of $\vartheta'$, the existence of the line bundle $L'$, extending $(\pi_* L_{|\mathcal{X}})^G$ and the nondegeneracy of $d\vartheta'$.
\end{defn}
Our goal now is to discuss conditions that imply the existence of the induced contact structure on the quotient. First, observe that clearly the existence of the descent of $L$ is a necessary condition, however the following example illustrates that it is not sufficient:
\begin{exa}\label{exa.quot}
	Let us consider $\C^4 \ni (x_0, x_1, x_2, x_3)$ with the symplectic form $\omega = dx_0 \wedge dx_2 + dx_1 \wedge dx_3$ and the action of $\Z_2$ with generator $g$, where $g \cdot (x_0, x_1, x_2, x_3) = (x_0, x_1, -x_2, -x_3)$. Then the projective space $\p^3 = \p(\C^4)$ has the (standard) contact structure $(F, L, \vartheta)$ and an action of $\Z_2$. The symplectic and contact forms are not fixed by the action: in fact we have $g \cdot \omega = -\omega$, $g \cdot \vartheta = -\vartheta$. At the same time, the line bundle $\mathcal{O}(-1)$ admits a linearization coming from the action on $\C^4$ and as a result, the action of the stabilizer of any point $x$ on $L_x  = \mathcal{O}(2)_x$ is trivial. Moreover, the kernel of $\vartheta$, i.e.~$F$ is preserved by the action. Consequently, the quotient $\p^3 / \Z_2$ admits a globally defined line bundle descended from $L$, a rank $2$ distribution defined on its smooth locus that is induced by $F$, yet they do not give the structure of a contact variety, as the map induced from $\vartheta$ is 0.
\end{exa}
We are now ready to state our criterion:
	\begin{thm}\label{thm.quotient}
	Let $(X,F,L, \vartheta)$ be a contact variety of dimension $2n+1 \ge 3$ and $G$ a finite group of automorphisms preserving $F$. The quotient $X/G$ has the induced contact structure $(F',L', \vartheta')$ if and only if the following two conditions are satisfied:
	\begin{enumerate}
		\item $\forall_{x \in X_{sing}}$ $\stab(x)$ acts trivially on $L_x$.
		\item $\forall_{g \in G}$ $g^* \vartheta = \vartheta$,
	\end{enumerate}
\end{thm}
\begin{proof}
	First, we will show that if the conditions are satisfied, then we can construct a contact structure on the quotient. Recall the result of Boutot \cite[Corollaire, p.2]{Bou87} who showed that a quotient of a variety with rational singularities by a reductive group action still has rational singularities. Let $\pi \colon X \rightarrow X/G$ be the quotient map. As before, $Z^g$ is the locus of points of $X_{sm}$ fixed by $g \in G \setminus \{ \id \}$ and let $Z = \bigcup_{g \in G \setminus \{\id \}} Z^g$. Any component of~$Z$ has codimension at least 2 by Lemma~\ref{lem.components}.\par
	Observe now that the listed conditions imply that $L$ descends to the quotient. Indeed, let $x \in X$ be any point with a nontrivial stabilizer $\stab(x) \subset G$. If $x\in X_{sm}$ then the invariance of the form implies that the action of $\stab(x)$ on $L_x$ is trivial. Indeed, we have a direct sum decomposition of $G$-modules: $T_x X = F_x \oplus L_x$ and $\vartheta_x$ is the projection onto the second summand. If there were $g \in \stab(x)$ acting nontrivially on $L_x$, then the projection from $T_xX$ onto $L_x$ (i.e.~$\vartheta$) would not commute with the action of $g$ and this is not the case. \par
	If $x \in X_{sing}$, then the triviality of the action of $\stab(x)$ on $L_x$ is directly assumed in the first listed condition and we claim the existence of the descent by the Kempf's lemma \cite[Thm~2.3]{DN89}. The resulting line bundle $L'$ will be the contact line bundle for $X/G$. \par
	The second listed condition implies that the map of the descended bundles induced from~$\vartheta$ is an epimorphism. Recall that $\mathcal{X}$ is precisely the preimage of $(X/G)_{sm}$. We have a sequence of maps $\pi^* T(X/G)_{sm} \simeq T\mathcal{X} \rightarrow L_{|\mathcal{X}} \simeq \pi^* L'_{|\mathcal{X}}$ that along with the invariance of $\vartheta$ allows us to put $\vartheta' \colon T(X/G)_{sm} \rightarrow L'_{|(X/G)_{sm}}$. We define the induced contact distribution $F'$ to be $\ker (\vartheta')$. \par 
	To conclude the proof in one direction, suppose that $\vartheta'$ is degenerate along some locus $B$, that necessarily is an effective divisor. Then we have $\mathcal{O}_{(X/G)_{sm}}(-K_{(X/G)_{sm}}) = (L'_{|(X/G_{sm})})^{\otimes(n+1)} \otimes \mathcal{O}_{(X/G)_{sm}}(-B)$ in the Picard group of $(X/G)_{sm}$. We pull back this relation to $\mathcal{X}$, but it can hold only if $\pi^* \mathcal{O}_{(X/G)_{sm}}(B)= \mathcal{O}_{\mathcal{X}} $, as $\mathcal{X}$ is a smooth subset of a contact variety, whose complement has codimension at least 2. Triviality of pullback implies that $\mathcal{O}_{(X/G)_{sm}}(-B)$ is a torsion element of the Picard group. But a nontrivial effective Weil divisor cannot give a torsion element in the Picard group, so we must have $B=0$. Consequently, $d\vartheta'$ is nowhere degenerate on the smooth locus and we conclude that if two listed conditions are satisfied, then the quotient has an induced contact structure. \par 
	Going the other way around, if we have a contact structure $(F', L', \vartheta')$ on $X/G$ induced from the one on $X$, that means in particular that $\pi^* L' = L$, so $L$ has a descent, so the second condition is satisfied. Moreover, $\vartheta$ on $X$ comes from the pullback of $\vartheta'$, so it is $G$-invariant.
\end{proof}
Observe that in the case when $X$ is smooth, the second assumption is void, so in particular we only need to check the invariance of $\vartheta$ to obtain a quotient variety that is contact. It will serve us to construct two examples of contact varieties, namely Example~\ref{exa.fav} and Example~\ref{exa.p5}. \par 
 One stark difference with the symplectic case is that to prove Theorem~\ref{thm.quotient} it was not enough to assume the invariance of the form and we needed an additional assumption on the triviality of the action on the fibers over the singular locus. Clearly, we use it to see that the contact line bundle descends to the quotient. One could hope that it is possible to get rid of this assumption by the use of the stratification theorem, if the induced contact forms on the strata were $G$-invariant. However, there is no reason to believe that this is the case and the twist of the contact form constitutes an unavoidable complication in the study of the contact varieties.
\section{Projective contact varieties}
From now on, we will restrict ourselves to considering varieties that are projective. Just like in the smooth case, the existence of a contact structure results in having negative Kodaira dimension and being uniruled. We also study the relation between birational morphisms that are crepant and those that preserve the contact structure, analogously to the folklore equivalence between symplectic and crepant resolutions. Finally, we construct some higher dimensional examples, classify projective threefolds and give an instance of a Campana-Flenner threefold that does not satisfy our definition.
\subsection{Uniruledness}
In the smooth case, the uniruledness of projective contact manifolds is a consequence of the theorem of Demailly. He proved some conditions that imply the integrability of a distribution defined as the kernel of the twisted form \cite[Main theorem]{Dem02}. The consequence of this result is that for a projective contact manifold the canonical divisor is not pseudoeffective \cite[Cor. 2]{Dem02}, so in turn such manifold is uniruled by \cite[Thm 2.2]{BDPP13}. The theorem of Demailly was very recently generalized by Cao and H\"oring \cite{CH22}, so that we can claim:
\begin{prop}\label{prop.not.psef}
	Let $(X,F,L,\vartheta)$ be a projective contact variety. Then $L^*$ is not pseudoeffective. Consequently, $X$ is uniruled and it admits a Mori contraction.
\end{prop}
\begin{proof}
	A projective contact variety is a normal compact K\"ahler space with klt singularities and a line bundle $L^*$ is a reflexive sheaf of rank one, so the assumptions of \cite[Thm 1.2]{CH22} are satisfied. In consequence, if $L^*$ were pseudoeffective, then the kernel $F$ of the twisted form $\vartheta \in H^0(X, (\Omega^1_X)^{**} \otimes L)$ would define a foliation, but it is absurd, as $F$ is non-integrable. \\
	As $L^*$ so also $K_X$ are not pseudoeffective, \cite[Thm 2.2]{BDPP13} implies that $X$ admits a moving family of rational curves intersecting negatively with $K_X$, so it is uniruled and the Mori cone has a nontrivial $K_X$-negative ray.
\end{proof}
In the singular case the Cone Theorem \cite[Thm 3.7]{KM98} states that any curve $C$ spanning an extremal ray has bounded intersection with $-K_X$. It can be translated to the condition that $L \cdot C \in \{1,2,3\}$ for the contact line bundle $L$. In the smooth case the bound is stronger and moreover, it allowed to give a full classification of possible Mori contractions of projective contact manifolds. In particular, it can be shown that they do not admit any birational Mori contractions \cite[Lem 2.10]{KPSW} that in turn gives a partial classification of manifolds themselves \cite[Thm 1.1]{KPSW}. Unfortunately, in our setting such result is currently out of reach. The reason is that for rational curves on singular varieties we cannot present the bound on the dimension of the space of morphisms \cite[Ch. I, Thm 2.16]{Ko96} in terms of cohomology. Consequently, the locus-fiber inequality \cite[Ch. IV, 2.6.1]{Ko96}, that is a fundamental tool in the proof of \cite[Thm 1.1]{KPSW}, does not take an easily computable form. Thus, we cannot determine the dimensions of the fibers, nor whether there are birational Mori contractions. However, it is easy to see that projective singular contact varieties that admit a crepant resolution have $L \cdot C = 1$ and the associated Mori contraction is of fiber type, induced from the projection $\p(TM) \rightarrow M$ of the resolving projective contact manifold. We will analyze such contractions during our classification of projective threefolds in Section~\ref{section.threefolds}.
	\subsection{Birational morphisms and resolutions of singularities}
	We begin our discussion by a slight strengthening of already mentioned result on the nonexistence of birational Mori contractions of projective contact manifolds \cite[Lem 2.10]{KPSW}.
		\begin{thm}\label{thm.contact.is.crepant}
		Let $f \colon \widetilde{X} \rightarrow X$ be a birational morphism from a projective contact \textbf{manifold} $\widetilde{X}$ to a variety $X$ with canonical and Gorenstein singularities (e.g. $X$ is a projective contact variety and $\widetilde{X}$ a resolution by a contact manifold). Then $f$ is crepant.
	\end{thm}
	\begin{proof}
	By the assumptions on $X$, we can write
	\begin{displaymath}
		K_{\widetilde{X}} = f^*K_X + \sum_i a_i E_i,
	\end{displaymath}
	where $a_i \ge 0$. Then $D = \sum a_i E_i$ corresponds to an effective Cartier divisor and $f_*(-D)=0$, so $f_*(-D)$ is also effective. Suppose that $D$ is $f$-nef. In such case we obtain by the negativity lemma \cite[Lem.~3.39]{KM98} that $-D$ is effective, so $D=0$ and we are done. \par 
	Now assume that $D$ is not $f$-nef, so there exists a curve $C$ contracted by $f$ that intersects $D$ negatively. We can write
	\begin{displaymath}
		K_{\widetilde{X}} \cdot C  = f^*K_X \cdot C + D \cdot C = 0 + D \cdot C.
	\end{displaymath} 
	It follows that $K_{\widetilde{X}} \cdot C < 0$, i.e.~$K_{\widetilde{X}}$ is not $f$-nef. By the Relative Cone Theorem \cite[Thm~3.25]{KM98} there exists a $K_{\widetilde{X}}$-negative ray, whose contraction factors $f$, however by \cite[Lem.~2.10]{KPSW} such maps cannot be birational. We reached a contradiction, so~$D$ must necessarily be $f$-nef and in this case we have our conclusion.
\end{proof}
	To provide a converse for this statement, we first need to discuss the notion of sheaf reflexivity. A coherent sheaf is \textit{reflexive} if it is isomorphic to its second dual. Hartshorne showed (\cite[Prop. 1.6]{Ha80}) that a sheaf $\mathcal{F}$ on a normal and integral scheme $Y$ is reflexive if and only if it is torsion free and satisfies the Barth normality condition, i.e.~for any open $U \subset Y$ and closed $Z \subset U$ of codimension at least 2 the restriction $\mathcal{F}(U) \rightarrow \mathcal{F}(U \setminus Z)$ is bijective, that is sections defined outside of $Z$ posses a unique extension to it. For a quasi-projective variety with canonical singularities, the pushforwards of sheaves of differentials from the resolution are reflexive by \cite[Thm 1.4]{GKKP11}. As the pushforward sheaf is defined over any open $U \subset X$ by $g_*\Omega^p_{\widetilde{X}}(U) = \Omega^p_{\widetilde{X}}(g^{-1}(U))$, then the reflexivity essentially means that if the codimension of the singular locus is at least 2, then any differential form on the smooth locus of $X$ extends uniquely to the everywhere defined differential form on the resolution. It allows us to prove the following:
	\begin{thm}\label{thm.crepant.is.contact}
		Let $(X,F,L, \vartheta)$ be a (quasi-)projective contact variety. Suppose that $f: X' \rightarrow X$ is a birational and crepant morphism. Then $X'$ is again a contact variety with the distribution $F'$ such that $f_* F' = F$ on the smooth locus of $X$ and the contact line bundle is $f^*L$. In particular, a terminalization of a contact variety is contact and a crepant resolution of singularities produces a classical contact manifold.
	\end{thm}
	\begin{proof}
		First observe that we have $f(\Exc(f)) \subset X_{sing}$. Indeed, if we had a component $Y \subset f(\Exc(f)) \cap X_{sm}$, then $X$ would be terminal and $\mathbb{Q}$-factorial in the neighbourhood of any point of the locally closed set $Y$, so $f$ could not be crepant.
		
		Let $g \colon \widetilde{X} \rightarrow X'$ be a resolution of singularities of $X'$, which composed with $f$ becomes a resolution of singularities of $X$. As $\vartheta$ is defined on the smooth locus of $X$, it can be considered as a section of the sheaf $((fg)_*\Omega^1_{\widetilde{X}}) \otimes L$. This sheaf is reflexive, as it is a product of $(fg)_* \Omega^1_{\widetilde{X}}$ that is reflexive by \cite[Thm 1.4]{GKKP11} and the line bundle $L$. Moreover, by the projection formula we can identify this product sheaf with  $(fg)_*(\Omega^1_{\widetilde{X}} \otimes (fg)^* L)$. Now, as $X$ is normal, its singular locus has codimension at least 2, therefore $\vartheta$ extends to a (global) section of $\Omega^1_{\widetilde{X}} \otimes (fg)^*L$. As $\widetilde{X}$ is a resolution of singularities of $X'$, we have an isomorphism $\widetilde{X} \setminus \Exc(g) \simeq X_{sm}'$, which allows us to define $\vartheta'$ on the smooth locus of $X'$. \\
		Now recall that $\vartheta \wedge (d\vartheta)^{\wedge n}$ can be extended to the whole $X$ and it is a nowhere vanishing section of $\mathcal{O}(K_X) \otimes L^{\otimes (n+1)}$. We pull back this section to $X'$, where thanks to $f$ being crepant it defines an isomorphism between $\mathcal{O}(-K_{X'})$ and $f^*L^{\otimes (n+1)}$. Moreover, on the intersection of $X_{sm}'$ and any open $U \subset X'$ trivializing $f^*L$ this section agrees with $\vartheta' \wedge (d\vartheta')^{\wedge n}$. \\
		Finally observe that $\vartheta'$ is surjective on $X_{sm}'$. If this were not the case at some point $x \in X_{sm}'$, then $\vartheta' \wedge (d\vartheta')^{\wedge n}$ would be zero at $x$, which is absurd. 
	\end{proof}
\subsection{Two higher dimensional examples}
To see an application of Theorem~\ref{thm.quotient} and Theorem~\ref{thm.crepant.is.contact} we will now construct two examples of contact varieties by the means of the finite quotient and discuss their particular resolutions. Interestingly enough, both admit an alternate description as unions of projectivized nilpotent orbits, that we will present next.
\begin{exa}[Author's favourite example]\label{exa.fav}
	Start with an affine space with fixed even dimension $\C^{2n+2}$ and a symplectic form: $dx_0 \wedge dx_{n+1} + ... + dx_n \wedge dx_{2n+1}$. Consider a finite group $\widehat{G}$ of symplectomorphisms generated by $\xi_i$ for $i=1,...,n$, where $\xi_i$ acts on a vector by multiplying its $i$-th and $i+n+1$-th coordinate by $(-1)$, i.e.:
	\begin{displaymath}
		\xi_i (x_0,...,x_{2n+1}) = (x_0,...,x_{i-1}, -x_i,x_{i+1},...,x_{i+n}, -x_{i+n+1},x_{i+n+2},...,x_{2n+1}).
	\end{displaymath}
	Note that $\widehat{G}$ is a subgroup of $T^{2n+2}$, i.e. the torus giving $\C^{2n+2}$ the structure of a toric variety. Now take the associated projective space $\p^{2n+1}$ equipped with an induced contact form $\vartheta = \sum_{i=0}^{n+1} (x_i dx_{i+n+1} - x_{i+n+1}dx_i)$. Observe that the action of $\widehat{G}$ descends to the projective space, it is a subgroup of the torus $T^{2n+1} \subset \p^{2n+1}$ and it preserves the twisted form. Denote the resulting group of contactomorphisms by $G \simeq \Z_2^n$ and consider the quotient $X = \p^{2n+1}/G$, which by Theorem \ref{thm.quotient} is a contact variety. \\
	To describe $X$ explicitly as a toric variety, let $M = \Hom(T^{2n+1}, \C^*)$ be the lattice of characters and $N \simeq \Z^{2n+1}$ the dual lattice spanned by $e_1,...,e_{2n+1}$. Then $\p^{2n+1}$ is described by $\Sigma \subset N_{\R}$. It is a complete fan having $\rho_1 = e_1,...,\rho_{2n+1} = e_{2n+1}, \rho_0 = -e_1 -... - e_{2n+1}$ as rays and such that every subset of rays is a cone. Now, finite subgroups of torus correspond bijectively to finite index sublattices $\iota: M' \hookrightarrow M$, so dually we have an inclusion $N \rightarrow N'$ with a finite cokernel (equal to $G$), so that $N_\R = N'_\R$. Therefore, to obtain the fan $\Sigma'$ of the quotient variety take the image of $\Sigma$ via the adjoint $\pi$ of the inclusion $\iota$. Our choice of generators of $M'$ will be: $w_i = e^*_i + e^*_{i+ n+1}$ and $w_{n+i+1} = e^*_i - e^*_{i+n+1}$ for$n =1,...,n$ and $w_{n+1} = e^*_{n+1}$, so that the matrix of $\iota$ is symmetric and in consequence it is also the matrix of the adjoint map $\pi$.\\
	The fan of $X$ is spanned by the following $2n+2$ rays: $\rho'_i = e_i + e_{i+ n+1}$ and $\rho'_{n+i+1} = e_i - e_{i+n+1}$ for $n =1,...,n$, $\rho'_{n+1} = e_{n+1}$ and $\rho'_0 = -2 \cdot (e_1 +... + e_n) - e_{n+1}$ and any subset of rays forms a cone. We denote a cone of the form $\cone(\rho'_i, \rho'_{n+i+1})$ for $i=0,...,n$ (mind the case $i=0$) by $\sigma_i$. Every such cone is singular as its generators cannot be extended to a basis of a whole lattice, and it corresponds to a codimension 2 singular subvariety of $X$, that is the image of $\p^{2n-1} = \{[x_0:...:x_{2n+1}] | x_i = x_{i+n+1} = 0\}$ via the quotient map.\\
	Toric resolution of singularities are provided by particular refinements of the fans. In our case, where every singular cone contains some $\sigma_i$ as a subcone, we first define $n+1$ new rays $\rho'_{E_i} = \frac{1}{2}(\rho'_i +\rho'_{i+n+1})$. Then we divide every $\sigma_i$ onto two cones: $\sigma'_i = \cone(\rho'_i, \rho'_{E_i})$ and $\sigma''_i = \cone(\rho'_{n+i+1}, \rho'_{E_i})$. We do analogously for every cone containing any $\sigma_i$ as a subcone. \\
	In this way we obtain a smooth projective toric variety $\widetilde{X}$ that maps onto $X$ and does not change the smooth locus, i.e.~a resolution of singularities. Moreover, one can compute that the pullback of $K_X$ is the canonical divisor of $\widetilde{X}$, so the resolution is crepant. Now, by \cite[Th\'eor\`eme, p.1]{Dr99} and Theorem~\ref{thm.crepant.is.contact} $\widetilde{X}$ has to be isomorphic to $\p(T(\p^1 \times ...\times \p^1))$. It can also be seen directly, as described in \cite[Prop. 7.3 and (7.6')]{Oda78}. We take the lattice $N'$ and project it by restricting to last $n+1$ coordinates. Then the image of $\Sigma'$ is a product fan for $n+1$ copies of $\p^1$, so $\widetilde{X}$ is equipped with a projection onto $\p^1 \times ... \times \p^1$. The kernel of the lattice projection contains a standard fan of $\p^{n}$, and this is a fiber of this projection. Toric computations show that this bundle is in fact the projectivization of $\mathcal{O}(2,...,2) = T\p^1 \times ... T\p^1$.
\end{exa}
Note that in higher dimensions the resolution that we have constructed is nonstandard. Indeed, according to the algorithm we should first blow up singular cones of higher dimension, i.e. maximal nontrivial intersections of $\sigma_i$, as such cones correspond to higher codimension components of singular locus. However, such resolution would not be crepant. \\
Additional significance of the example above comes from the fact that it allows us to provide a link between $\p^{2n+1}$ and $\p(T(\p^1 \times ... \times \p^1))$, i.e. two disjoint families of contact manifolds via our notion of a contact variety. It is an interesting question whether other contact manifolds can be linked similarly. 
\begin{exa}\label{exa.p5}
	Now take $X= \p^5$ with the contact and toric structure as described in the previous example and let $\xi$ be the generator of $\Z_2$. Assume that $\Z_2$ acts on $X$ by: $\xi \cdot [x_0:x_1:x_2:x_3:x_4:x_5] = [x_0:-x_1:x_2:x_3:-x_4:x_5]$ and consider the quotient of $X$ by this action. Then the fixed point locus has two components, described by $x_1 = x_4 = 0$ ($\p^3$) and $x_0 =x_2 =x_3= x_5 = 0$ ($\p^1$), which get mapped isomorphically onto two components of the singular locus in the quotient. The (toric) resolution is provided by two disjoint blow-ups centered at those components. One can check that the resulting smooth variety $\widetilde{X}$ has no chance of being contact, as its canonical divisor is not divisible by 3 in the class group. Nevertheless, the partial resolution obtained by blowing up just the bigger component is crepant, so it is another projective contact variety, call it $X'$. One can verify that $\widetilde{X}$ admits a twisted form that maps $T\widetilde{X}$ onto the pullback of the contact line bundle, but the nondegeneracy condition is satisfied on some open subset (equal to the complement of the exceptional divisor). Such varieties are sometimes called \textit{generically contact}.
\end{exa}
On the other hand, the second example shows that in dimensions higher than 3 it is not possible to reduce the study of projective contact varieties to the analysis of crepant contractions of projective contact manifolds. 
\begin{rmk}\label{rmk.example}
	The language of nilpotent orbits allows us to give an alternate description of both described examples. To begin with, consider the simple group $\SL(2)$. Its Lie algebra $\mathfrak{sl}(2)$ is 3 dimensional and has the 2-dimensional nilpotent cone:
	\begin{displaymath}
		N = \left\{ \begin{bmatrix}
			x & y \\
			z  &-x
		\end{bmatrix} \mid x^2 +yz =0 \right\} = \SL(2) \cdot E_{1,2},
	\end{displaymath}
	that -- when we pass to the projectivization of the algebra -- gets mapped to $\p^1$, that is simultaneously the minimal and the principal orbit. Now take a product of $(n+1)$ copies of $\SL(2)$. It is a semisimple group whose algebra is a direct sum of $(n+1)$ copies of $\mathfrak{sl}(2)$ and the adjoint action is component-wise. In particular, the nilpotent cone is the sum of nilpotent cones of the components. Every nilpotent orbit has generator of the form $(m_0,m_1,...,m_n)$ where each $m_i$ is either the zero matrix or the elementary matrix $E_{1,2}$. Clearly, the element $(E_{1,2},...,E_{1,2})$ belongs to the principal orbit. Moreover, for an orbit given by a generator we can obtain generators of orbits lying in its closure by replacing some nonzero $m_i$'s by zero matrices. \par
	To see that the projectivization of the nilpotent cone coincides with the toric quotient variety described in Example~\ref{exa.fav}, let us define a map:
	\begin{displaymath}
		\C^{2n+2} \rightarrow \mathfrak{sl}(2) \oplus ... \oplus \mathfrak{sl}(2),
	\end{displaymath}
	\begin{displaymath}
		(x_0,x_1,...,x_{2n+1}) \mapsto \left( \begin{bmatrix} x_0 \cdot x_{n+1} & x_0^2 \\ -x_{n+1}^2 & -x_0 \cdot x_{n+1} \end{bmatrix},...,  \begin{bmatrix} x_n \cdot x_{2n+1} & x_n^2 \\ -x_{2n+1}^2 & -x_n \cdot x_{2n+1} \end{bmatrix} \right).
	\end{displaymath}
	Now observe that the image of the affine space is precisely the nilpotent cone and that the map descends to the morphism between the projectivizations (of the affine space and of the algebra). Moreover, recall that in Example~\ref{exa.fav} we have defined an action of the group $\mathbb{Z}_2 \times ...\times \mathbb{Z}_2$ on $\p^{2n+1}$ and we see that the projectivization of our map is constant on orbits of this action. Finally, by a direct computation we may verify that the symplectic forms agree and in this way we rediscover our toric quotient as a projectivized nilpotent cone. Moreover, the resolution that we have described is the projectivized version of \textit{the Springer resolution} (see \cite[Section~6]{Gin97} or \cite{Fu03} for a reference). \par
	Example~\ref{exa.p5} is a union of 3 orbits in the projectivization of the algebra $\mathfrak{sp}(4) \oplus \mathfrak{sl}(2)$. Recall that for the symplectic group the adjoint variety is the projective space, embedded via the Veronese map. We put:
	\begin{displaymath}
		\C^{6} \rightarrow \mathfrak{sp}(4) \oplus \mathfrak{sl}(2),
	\end{displaymath}
	\begin{displaymath}
		(x_0,x_1,x_2,x_3,x_4,x_5) \mapsto \left( \begin{bmatrix} x_0 x_5 & x_0 x_3 & x_0^2 & x_0 x_2 \\ x_2 x_5 & x_2 x_3 & x_0 x_2 & x_2^2 \\ x_3^2 & x_3 x_5 & -x_0 x_5 & -x_2 x_5 \\ x_3 x_5 & x_5^2 & -x_0 x_3 & -x_2 x_3 \end{bmatrix}, \begin{bmatrix} x_1 x_4 & x_1^2 \\ -x_{4}^2 & -x_1 x_4 \end{bmatrix} \right).
	\end{displaymath}
	One can see that the $4 \times 4$ matrix is indeed an element of the algebra $\mathfrak{sp}(4)$: the off-diagonal $2 \times 2$ blocks are symmetric and the lower diagonal block is the negative transpose of the upper diagonal block. Moreover, it is traceless and of rank 1, so it is nilpotent and more precisely belongs to the minimal (nilpotent) orbit $\mathcal{O}_{ [ 2,1,1] }$. As before, we pass to the map between projective spaces and observe that it is constant on $\Z_2$-orbits to conclude. The two components of the singular locus described in Example~\ref{exa.p5} are precisely the projectivizations of the two minimal orbits of $\mathfrak{sp}(4)$ and $\mathfrak{sl}(2)$.
\end{rmk}
	\subsection{Projective singular contact threefolds}\label{section.threefolds}
	As a conclusion of our study we give the full classification of strictly singular projective threefolds equipped with the contact structure. It turns out that they all can be constructed from ruled surfaces, therefore we begin with fixing the notation.
\begin{set}\label{set.ruled}
	In this section $S$ denotes a ruled surface in the sense of \cite[V.2]{Ha77}, i.e.~a smooth projective surface equipped with a surjective morphism $p \colon S \rightarrow B$ to a smooth projective curve $B$ of genus $g$ such that every fiber is isomorphic to $\p^1$. We will denote (any) such fiber by $\ell$, remembering that it can also be considered as an effective divisor on $S$. With these assumptions, there exists a (non-unique) locally free sheaf $\mathcal{E}$ of rank 2 on $B$ such that $S \simeq \p(\mathcal{E})$ \cite[V, Prop.~2.2]{Ha77}. Moreover, to normalize the construction, one can demand that $H^0(B,\mathcal{E}) \ne 0$ but $H^0(B,\mathcal{E} \otimes \mathcal{L}) = 0$ for any line bundle $\mathcal{L}$ on $B$ of negative degree. Then $e = -\deg \mathcal{E}$ is an invariant of $S$ and there exists a section $s \colon B \rightarrow S$ such that the (divisor) class of its image, $B_0$ is equal to $\mathcal{O}_{\p(\mathcal{E})} (1)$ \cite[V, Prop.~2.8]{Ha77}. The space $N_1(S)$ as well as $N^1(S)$ is spanned by the classes of $\ell$ and $B_0$. \par
	In the particular case when $B = \p^1$, there are countably many isomorphism classes of ruled surfaces, determined by the integer $e \ge 1$, that will be denoted by $\mathbb{F}_e \coloneqq \p(\mathcal{O}_{\p^1} \oplus~\mathcal{O}_{\p^1}(-e))$. They are known as \textit{the Hirzebruch surfaces}. Every Hirzebruch surface admits a divisorial contraction of the section $B_0$ (it is called \textit{the minimal section}, as any other section has larger self-intersection) and the resulting variety is a cone that we will denote by $\mathcal{C}_e$. \par 
	Finally, we will consider the projectivized tangent bundle $\p(TS)$ over $S$, along with the natural projection $\pi \colon \p(TS) \rightarrow S$ and the fiber $C_{\pi}$. For the brevity of notation, we will sometimes denote by $\xi$ the class of $\mathcal{O}_{\p(TS)}(1)$. In our setting, $N^1(\p(TS))$ is spanned by classes of $\xi$, $\pi^* \ell$ and $\pi^* C_0$ \cite[II, Ex.~7.9]{Ha77}. Note that although in general $TS$ could also be twisted by some line bundle and the resulting projectivized bundle would be isomorphic to $\p(TS)$, we do not do it, as we want to keep the natural surjective morphism $\pi^* TS \rightarrow \mathcal{O}_{\p(TS)}(1)$.
\end{set}
\begin{rmk}
	There exists a unique ruled surface admitting two distinct rulings, namely $\p^1 \times \p^1$. It will provide a special case in our classification and it will require some additional care in our reasonings.
\end{rmk}
Our goal is to prove the following two theorems:
\begin{thm}\label{thm.threefolds}
	Notation as in Setting \ref{set.ruled}. Let $(X,L)$ be a singular contact variety in dimension 3 that is projective. Then $X$ admits a crepant resolution $f \colon \p(TS) \rightarrow X$ for some $S$. Going the other way around, every $\p(TS)$ (over a base curve $B$) admits a section $\sigma \colon S \rightarrow \p(TS)$ and a crepant morphism onto a singular contact threefold $X$. This crepant morphism contracts the image of $\sigma$ onto a curve isomorphic to $B$. Moreover, if $X$ is not Fano, then $\rho(X) =2$ and $X$ is a locally trivial bundle over $B$, whose fibers are cones $\mathcal{C}_2$. In this case, we have the following commutative diagram of contractions and sections:
	\begin{center}
		\begin{tikzcd}
			\p(TS) \arrow[r, "f"] \arrow[d, "\pi"] & X \arrow[d, "\pi'"] \\
			S \arrow[r, "p"] \arrow[u, bend left, "\sigma"] & B \arrow[l, bend left, "s"].
		\end{tikzcd}
	\end{center}
\end{thm}

\begin{thm}\label{thm.3fold.fano}
	There exists a unique singular projective Fano contact variety $X$ in dimension 3. It has $\rho(X) =1$ and is resolved by $\p(T(\p^1 \times \p^1))$. The resolution morphism is given by the contraction associated to the linear system $| - b K_{\p(T(\p^1 \times \p^1))}|$ for $b >> 0$. Moreover, it coincides with the variety described in Example~\ref{exa.fav} and Remark~\ref{rmk.example} for dimension 3, i.e.~it can be described either as a quotient of $\p^3$ by $\Z_2$ or as a projectivization of the nilpotent cone of $\mathfrak{sl}(2) \times \mathfrak{sl}(2) $.
\end{thm}
\begin{rmk}
	Note that contact manifolds of the form $\p(TS)$ may admit different contact forms, and in fact we have a bijection (\cite[Prop.~2.14]{KPSW}):
	\begin{displaymath}
		H^0(S, \End(\Omega^1_S)) \rightarrow H^0(\p(TS), \Omega^1_{\p(TS)}(1))
	\end{displaymath}
	between automorphisms of $\Omega^1_S$ and contact forms on $\p(TS)$. In light of this identification, we should not expect that the contact structure on $X$ is unique.
\end{rmk}
We split the proofs of both theorems into a few auxiliaries. First, we will show that our threefolds can always be resolved by a projective contact manifold.
\begin{prop}\label{prop.ruled}
	Let $(X,L)$ be a projective singular contact threefold. Then $X$ has a crepant resolution of singularities by the projective contact manifold $(\p(TS), \mathcal{O}_{\p(TS)}(1))$ for some $S$ as in Setting~\ref{set.ruled}. Moreover, a ruling on $S$ is given by the image (via $\pi$) of the rational curve contracted by the resolution morphism.
\end{prop}
\begin{proof}
	To begin with, we show that $X$ admits a crepant resolution of singularities. Indeed, let us denote by $f\colon \widetilde{X} \rightarrow X$ the crepant terminalization of $X$, that exists by \cite[Cor.~1.4.4]{BCHM06}. It is a contact variety by Theorem~\ref{thm.crepant.is.contact}. The singular locus of a terminal variety has codimension at least 3 by \cite[Lem.~1.3.1]{BS95}, but on the other hand by Corollary~\ref{cor.stratification} its codimension must be even, therefore it is empty and $f$ already resolves all singularities and is crepant. Moreover, Corollary~\ref{cor.stratification} shows that the singular locus consists of disjoint, smooth curves. As $\widetilde{X}$ is a projective contact manifold that is not Fano (it admits a crepant contraction), it must be isomorphic to $\p(T\Sigma)$ with the contact line bundle $\mathcal{O}_{\widetilde{X}}(1)$ for some smooth projective surface $\Sigma$ by \cite[Thm~1.1]{KPSW} and moreover $f^* L = \mathcal{O}_{\widetilde{X}}(1)$. \par
	Now we will show that $\Sigma$ has to be a ruled surface. To that end, recall that every fiber of $f$ is covered by rational curves by \cite[Cor.~1.6]{HM07}. Denote by $E$ an irreducible (divisorial) component of the exceptional locus of $f$ that gets mapped onto some curve $C_X$ in the singular locus of $X$. We have a rational curve $C_E \subset E$ such that $C_E \cdot  \mathcal{O}_{\widetilde{X}}(1) = 0$, as the morphism contracting $C_E$ is crepant and $\mathcal{O}_{\widetilde{X}}(-K_{\widetilde{X}}) = \mathcal{O}_{\widetilde{X}}(2)$. Denote by $\gamma: \p^1 \rightarrow C_E \subset \p(T\Sigma)$ the normalization of $C_E$. Consider the relative Euler sequence (\cite[Lem.~2.5]{KPSW}):
	\begin{displaymath}
		0 \rightarrow \Omega^1_{\p(T\Sigma)/\Sigma} \otimes \mathcal{O}_{\widetilde{X}}(1) \rightarrow \pi^* T\Sigma \rightarrow \mathcal{O}_{\widetilde{X}}(1) \rightarrow 0
	\end{displaymath}
	and pull it back via $\gamma$, remembering that every vector bundle on $\p^1$ splits and that $\mathcal{O}_{\widetilde{X}}(1)$ is trivial on $C_E$. We obtain:
	\begin{displaymath}
		0 \rightarrow \mathcal{O}_{\p^1}(b) \rightarrow \mathcal{O}_{\p^1}(a_1) \oplus \mathcal{O}_{\p^1}(a_2) \rightarrow \mathcal{O}_{\p^1} \rightarrow 0,
	\end{displaymath}
	where we know that $a_1 \ge 2$ by \cite[II,~Lemma~3.13]{Ko96}. Consequently, the map $ \mathcal{O}_{\p^1}(a_1)\rightarrow \mathcal{O}_{\p^1}$ must be trivial, so we have $a_2 = 0$. It follows that there exists a free rational curve on $\Sigma$, so it is uniruled. Moreover, the numerical class of $C_E$ belongs to a $K_{\p(T\Sigma)}$-trivial extremal ray in the Mori cone of $\p(TS)$, so $\pi(C_E)$ belongs to a $K_\Sigma$-negative extremal ray in the Mori cone of $\Sigma$. Let $C$ be a rational curve on $\Sigma$ spanning the latter ray. By \cite[Thm~1.28]{KM98} we have $-K_{\Sigma} \cdot C \in \{1,2,3\}$, so via the adjunction formula we obtain $C^2 \in \{-1,0,1\}$ and these three possible self intersections correspond to three different contractions:
	\begin{itemize}
		\item If $C^2 = -1$, then $C$ is a smooth contractible $(-1)$-curve and it particular it does not move,
		\item if $C^2 = 0$, then $\Sigma$ is a ruled surface $S$ and $C$ is its fiber $\ell$, 
		\item if $C^2 = 1$, then $\Sigma \simeq \p^2$.
	\end{itemize} 
	We easily exclude the last possibility, as $\p(T\p^2)$ is a homogeneous Fano manifold, so in particular it does not admit any curves intersecting trivially with the anticanonical, consequently $C_E$ could not exist in this case. \par 
	To see that the first case is also impossible, observe that if $C$ does not move, then we necessarily have $\pi(E)  = C$, so $E$ is a ruled surface $\p(T\Sigma_{|C})$ over $C$. Now consider two maps from $E$, namely $\pi_{|E}$ onto $C$ and $f_{|E}$ onto $C_X$. Note that both have connected fibers and by the Rigidity Lemma~\cite[Lem.~1.6]{KM98} no fiber of $\pi_{|E}$ gets contracted to a point by $f_{|E}$ and vice versa. In particular, rational curves that are fibers of $\pi_{|E}$ are mapped onto $C_X$, so we have $C_X \simeq \p^1$. But this means that $E$ is a trivial ruled surface over $C \simeq \p^1$, so $T\Sigma_{|C} \simeq \mathcal{O}_{C}(a) \oplus \mathcal{O}_C (a)$ for some $a \in \mathbb{Z}$, as $E$ is not necessarily a normalized ruled surface. We have reached a contradiction, as we have $-K_{\Sigma} \cdot C = 1 \ne 2a$ for any $a$. We conclude that $\Sigma$ is indeed a ruled surface $S$. \par
	The last statement is clear: the resolution morphism contracts $C_E$ and we have just argued that $\pi(C_E)$ is the curve whose numerical class is a multiple of a ruling on $S$.
\end{proof}
Now we will establish the existence of section for $\pi$ and show some consequences of it.
\begin{prop}\label{prop.section}
	Notation as in Setting \ref{set.ruled}. A ruling $p \colon S \rightarrow B$ induces a section $\sigma \colon S \rightarrow \p(TS)$ that in particular allows us to embed $N_1(S)$ in $N_1(\p(TS))$. We have the following intersection table:
	\begin{center}
		\begin{tabular}{|c|c|c|c|}
			\hline 
			$\cdot$ & $\sigma_* [\ell]$ & $\sigma_* [B_0]$ & $[C_\pi]$ \\ \hline
			$\pi^* \ell$ & 0 & 1 & 0 \\ \hline
			$\pi^* B_0$ & 1 & $-e $& 0 \\ \hline
			$\mathcal{O}_{\p(TS)}(1)$ & 0 & $2-2g$ & 1 \\ \hline
		\end{tabular}
	\end{center}
	Consequently, the classes of curves $\sigma_* [\ell], \sigma_* [B_0], [C_\pi]$ form a basis of the vector space $N_1(\p(TS))$. For any $S$, the Mori cone of $\p(TS)$ has a face containing $C_{\pi}$ and $\sigma_* [B_0]$ that intersects nontrivially with the hypersurface $K_{\p(TS)} = 0$.
\end{prop}
\begin{proof}
	The map $p \colon S \rightarrow B$ induces an epimorphism $TS \twoheadrightarrow p^*TB$ that corresponds to the section $\sigma$ by \cite[II, Prop.~7.12]{Ha77} and it holds that $\pi \circ \sigma = \id_{S}$. In consequence, we have the following maps of vector spaces:
	\begin{center}
		\begin{tikzcd}
			N_1 (S) \arrow[r, "\sigma_*"] & N_1(\p(TS)) \arrow[r, "\pi_*"] & N_1(S) \\
			N^1(S) \arrow[r, "\pi^*"] & N^1(\p(TS))  \arrow[r, "\sigma^*"] & N^1(S),
		\end{tikzcd}
	\end{center}
	where both compositions are identity. \par
	To compute two first rows of the intersection table, simply use the projection formula, remembering that on $S$ we have $\ell \cdot \ell = 0$, $\ell \cdot B_0 = 1$ and $B_0 \cdot B_0 = -e$ (the last equality is shown in \cite[Ch.~V, Prop.~2.9]{Ha77}). Moreover, $\pi_* [C_\pi] = 0$, so $C_\pi$ necessarily intersects trivially with divisors pulled back from $S$. \par
	Now we will consider the last row. We have $\mathcal{O}_{\p(TS)}(1) \cdot C_\pi =1$ by definition.
	If $C$ is any curve on $S$, then to compute its intersection with $\mathcal{O}_{\p(TS)}(1)$ we again employ the projection formula, this time for the morphism $\sigma$:
	\begin{displaymath}
		\sigma_*[C] \cdot \mathcal{O}_{\p(TS)}(1) = [C]  \cdot  \sigma^*(\mathcal{O}_{\p(TS)}(1) ) = [C] \cdot p^* TB,
	\end{displaymath}
	where the last equality comes from \cite[II,~Prop.~7.12]{Ha77}. In particular, we have $[\ell] \cdot p^* TB = 0$ and $[B_0] \cdot p^* TB = 2- 2g$. Recall from our earlier discussion in Setting~\ref{set.ruled} that $\mathcal{O}_{\p(TS)}(2) = \mathcal{O}(-K_{\p(TS)})$, so the morphism contracting $\sigma_* [\ell]$ -- if exists -- is crepant.\par  
		For the claim on the basis, the space $N_1(\p(TS))$ is of dimension 3 as a dual of $N^1(\p(TS))$ for which we have already picked a base (in Setting~\ref{set.ruled}). The computations of the intersection table show clearly that picked representatives are linearly independent in $N_1(\p(TS))$, so they form a base. \par 
		For the last claim, first observe that $\sigma_* [B_0]$ and $[C_\pi]$ are contained in a face of $\overline{NE(\p(TS))}$. Indeed, we have assumed in Setting~\ref{set.ruled} that $S$ is normalized, so in particular $[B_0]$ is a ray bounding the cone $\overline{NE(S)}$ and $\sigma_*[B_0]$ gets mapped onto this ray by $\pi_*$, so it cannot lie inside of the cone.
		
		Now, if $g(B_0)\ge 1$ then we are done, as in those cases $\sigma_* [B_0]$ intersects $-K_{\p(TS)}$ nonpositively. This is not the case if $g(B_0) = 0$, but since by \cite[Prop.~2.13]{KPSW} the cone $\overline{NE(\p(TS))}$ admits no $K_{\p(TS)}$-negative extremal rays other than $[C_\pi]$, then we deduce that the class of $\sigma_*[B_0] -2 [C_\pi]$ must also belong to the face containing  $\sigma_* [B_0]$ and $[C_\pi]$. In fact, one can directly show the existence of a curve in this class by considering the ruled surface $\p(TS_{|B_0})$, but we won't need it.

\end{proof}
We will make use of the following self-intersection:
\begin{lem}\label{lem.intersections}
	Notation as in Setting~\ref{set.ruled}. We have:
	\begin{displaymath}
		\xi^3 = 4(1-g).
	\end{displaymath}
\end{lem}
\begin{proof}
	To begin with, recall that the pullback of cocycles $\pi^*$ gives the Chow ring $A(\p(TS))$ the structure of a free $A(S)$-module with a basis $1, \xi$ (\cite[App. A, 2.A.11]{Ha77}), where $\xi$ is the class of $\mathcal{O}_{\p(TS)}(1)$, so that $2 \xi$ is the class of $-K_{\p(TS)}$. Moreover, we have the following relation (\cite[App. A, 3, Definition on p. 429]{Ha77}):
	\begin{equation}\label{eqtn.chow}
		\xi^2 - \xi \cdot \pi^* c_1(TS) + \pi^* c_2(TS) = 0.
	\end{equation}
	Our task is to compute:
	\begin{align*}
		\xi^3  =\cdot \xi \cdot \xi^2 =  \xi (\xi \cdot \pi^* c_1(TS) -\pi^* c_2(TS))= \\
		=  (\pi^* c_1(TS))^2 - \pi^* c_1(TS) \cdot \pi^* c_2(TS) - \xi \cdot \pi^* c_2(TS) = \\
		= \xi \cdot (\pi^* c_1(TS))^2 - \xi \cdot \pi^* c_2(TS).
	\end{align*}
	We have $c_1(TS)^2 = 8(1-g)$ by \cite[V, Cor.~2.11]{Ha77}. Recall the Riemann-Roch formula for surfaces \cite[V, Remark 1.6.1]{Ha77}:
	\begin{displaymath}
		12(1+ p_a) = c_1(TS)^2 + c_2(TS),
	\end{displaymath}
	where $p_a$ is the arithmetic genus that is equal to $-g$ \cite[V, Cor.~2.5]{Ha77}. It allows us to compute that $c_2(TS) = 4(1-g)$ and consequently we easily obtain the desired intersection number ($[x]$ denotes the class of a point):
	\begin{displaymath}
		\xi^3 = \xi \cdot 4(1-g)\pi^* [x] = 4(1-g).
	\end{displaymath}
\end{proof}
We are ready to prove the existence of the contraction:
\begin{prop}\label{prop.shok}
	Notation as in Setting~\ref{set.ruled}. For any $\p(TS)$ and a positive integer $a$ such that $a > 3\cdot \max \{ 2g-2, e\}$, the linear system of some positive multiple of the divisor $D_a = a \cdot \pi^* \ell + \xi$ gives a crepant and birational morphism, contracting precisely the class $\sigma_* [\ell]$. In the special case when $S = \p^1 \times \p^1$ there is another crepant and birational morphism given by the linear system of some positive multiple of $-K_{\p(T(\p^1 \times \p^1))}$ (or $\xi$). It contract images of two sections, corresponding to two distinct rulings on $\p^1 \times \p^1$.
\end{prop}
\begin{proof}
	To show that the linear system of some positive multiple of a divisor $D$ gives a birational morphism, it is enough to show that $D$ is big, nef and semiample. To that end, consider a vector bundle $TS \otimes \mathcal{O}_S(a\cdot \ell)$. First, we will show that it is nef and big\footnote{The author is grateful to the anonymous referee for suggesting this argument, much clearer than the one originally provided.} (in the sense of Lazarsfeld) for $a$ large enough, as by definition it means that $D_a$ is nef and big. We have the twisted relative tangent sequence:
		\begin{equation}\label{eqtn.ses}
		0 \rightarrow T_{S/B} \otimes \mathcal{O}_S(a \cdot \ell) \rightarrow TS \otimes \mathcal{O}_S(a \cdot \ell) \rightarrow p^* TB  \otimes \mathcal{O}_S(a \cdot \ell) \rightarrow 0.
	\end{equation}
	From the relative Euler sequence for $S = \p(\mathcal{E})$:
	\begin{displaymath}
		0 \rightarrow \mathcal{O}_S \rightarrow p^* (\mathcal{E}^*) \otimes \mathcal{O}_S(1) \rightarrow T_{S/B} \rightarrow 0
	\end{displaymath}
	we obtain $T_{S/B} \simeq p^*(\det \mathcal{E}^*) \otimes  \mathcal{O}_S(2)$, so numerically we have $T_{S/B} \equiv 2B_0 + e \cdot \ell$ and $ p^* TB \equiv (2 -2g) \cdot \ell$. \par 
	To show that the vector bundle in the middle of the sequence~(\ref{eqtn.ses}) is nef, by \cite[Prop.~1.2.(5)]{CP91} it is enough to show that the other two terms in the short exact sequence are nef. This is straightforward: the term on the right is nef if it is numerically equivalent to a nonnegative multiple of $\ell$ and it holds if $a \ge 2g -2$. The term on the left is nef if $a \ge e$. Consequently if $a \ge \max \{ 2g-2, e\}$ then we ensure the nefness of $D_a$. Note that it follows that in the special case $D_0 = \xi$ is also nef. \par 
	Now we will show that $TS \otimes \mathcal{O}_S(a \cdot \ell)$ is big. Observe that it has a subbundle $T_{S/B} \otimes \mathcal{O}_S(a \cdot \ell) \equiv 2B_0 +(a+ e)\ell$ that is ample for $a > e$, so the bigness of $D_a$ follows from \cite[Lem.~2.4]{Liu23}. In the case when $S = \p^1 \times \p^1$ we cannot argue in the same way, but $-K_{\p(T(\p^1 \times \p^1))}$ is big as by Lemma~\ref{lem.intersections} its top self-intersection is positive. \par 
	To prove the semiampleness, simply invoke the Shokurov's Theorem \cite[Thm~3.3]{KM98}: we have just shown that $D_a$ is big and nef. If we replace $a$ by $3a$ then we can ensure that $D_a - K_{\p(TS)}$ is also nef and big, so the assumptions are satisfied and we claim the semiampleness of $D_a$. In the special case we also use the Shokurov's Theorem, this time for the divisor $D_0 = \xi$ that was also shown to be big and nef. \par 
	Consequently, we have shown that the map given by the linear system of some positive multiple $D_a$ for $a > 3 \cdot \max \{e, 2g-2\}$ is a birational morphism. Moreover, $D_a \cdot \sigma_*[\ell] = 0$ from Proposition~\ref{prop.section} and since already $D_a - K_{\p(TS)}$ is nef, we observe that $\sigma_*[\ell]$ is the only class intersecting $D_a$ trivially, so the morphism contracts only this class. Then, as $K_{\p(TS)} \cdot \sigma_*[\ell] = 0$, the morphism is crepant and we are done in the general case. \par 
	If $S = \p^1 \times \p^1$, then the map given by the anticanonical linear system is a birational morphism that is clearly crepant. As shown in Proposition~\ref{prop.section}, any ruling $\ell_i$ on $S$ gives a section $\sigma_i$ and we have $\sigma_i(\ell_i) \cdot K_{\p(TS)} = 0 $, so the morphism given by the anticanonical contracts chosen rulings on both sections.
\end{proof}
We are ready to finish proofs of both main theorems, starting from Theorem~\ref{thm.3fold.fano}. To do it quickly, we will use the classical notion of a Campana-Peternell manifold. A projective manifold $M$ is Campana-Peternell (CP) if it has nef tangent bundle (i.e.~such that $\mathcal{O}_{\p (TM)}(1)$ is nef). Conjecturally, in the Fano case CP manifolds are rational homogeneous spaces, a characterization extending the one proved by Mori for the projective space. Moreover, all CP surfaces were classified in \cite[Thm~3.1]{CP91}. In particular, for a ruled surface $S$ over $B$, $\mathcal{O}_{\p(TS)}(1)$ is nef if and only if $S = \p^1 \times \p^1$ (we have just observed that) or $B$ is an elliptic curve and the vector bundle $\mathcal{E}$ is semistable. 
\begin{proof}[Proof of Theorem~\ref{thm.3fold.fano}]
	Let $X$ be a contact Fano variety of dimension 3 with a nonempty singular locus. By Proposition~\ref{prop.ruled} it admits a crepant resolution of singularities by $\p(TS)$ for some $S$. The anticanonical divisor $-K_X$ is ample by definition, so $-K_{\p(TS)} = f^* (-K_X)$ is nef and moreover we may take $-K_{\p(TS)}$ to be the supporting divisor for the contraction $f$, as it is the pullback of an ample divisor from the target. The nefness of $\mathcal{O}(-K_{\p(TS)}) = \mathcal{O}_{\p(TS)}(2)$ implies that $S$ is Campana-Peternell and in particular $B$ is either a projective line or an elliptic curve. We exclude the second case, as then $\xi^3 =0$ by Lemma~\ref{lem.intersections}, so the resulting map would not be birational. We conclude that $S= \p^1 \times \p^1$. Moreover, as we have already observed in Proposition~\ref{prop.shok}, the linear system $|-b\cdot K_{\p(TS)}|$ for $b >> 0$ does indeed define a birational morphism that contracts two sections of $\pi$ along their chosen rulings. \par
	To see that $\rho(X) = 1$ it is enough to observe that the Mori cone of $\p(T(\p^1 \times \p^1))$ is spanned by three rays, namely $[C_\pi]$, $[\sigma_1(\ell_1)]$ and $[\sigma_2(\ell_2)]$ and the two latter are contracted by $f$. \par 
	The prove the last claim simply observe that the variety described in Example~\ref{exa.fav} is Fano for every possible dimension (it can be done either by a direct toric computation or invoking adequate result from the theory of nilpotent orbits) and has a nonempty singular locus (we have already described it).
\end{proof}
\begin{proof}[Proof of Theorem~\ref{thm.threefolds}]
	Let $X$ be a projective singular contact variety of dimension 3. By Proposition~\ref{prop.ruled} it has a crepant resolution of singularities of the form $\p(TS)$ and the resolution morphism contracts the class of a rational curve $C$ such that $\pi_*[C] = [\ell]$. If $S = \p^1 \times \p^1$ and $\rho(X) =1$, then $X$ is Fano, so we may exclude this case from our current reasoning, as it was discussed in Theorem~\ref{thm.3fold.fano}. In particular, it means that $C$ distinguishes one ruling if $S$ admits two, that in turn determines the section $\sigma \colon S \rightarrow \p(TS)$. Then Proposition~\ref{prop.section} shows that we can identify the ray spanned by $[C]$ with the one spanned by $\sigma_* [\ell]$, as both lie in the intersection of the hyperplanes $K_{\p(TS)} = 0$ and $\pi^* \ell = 0 $. In Proposition~\ref{prop.shok} we have proven that the birational morphism contracting $\sigma_*[\ell]$ is given by the linear system of some positive multiple of $D_{a} = a \cdot \pi^* \ell + \xi$ for some integer $a >0$. \par
	Now we want to prove that for every $\p(TS)$ we can contract the section $\sigma$ and obtain a contact variety. To that end, let us denote by $\phi$ the morphism given by $|b \cdot D_{a}|$ for $b >>0 $ and by $Y$ the target variety. We need to show that there exists a commutative square of contractions, $Y$ is a locally trivial bundle of cones with $\rho(Y) = 2$ and that it admits a contact structure. 
	
	First, let us observe that there are two distinct classes of curves on $Y$, whose representatives are $\phi(C_\pi)$ and $\phi(\sigma(B_0))$, so in particular $\rho(Y) = 2$ and the Mori cone $\overline{NE(Y)}$ is spanned by two rays. Moreover, $\phi$ induces a surjective map $\overline{NE(\p(TS))} \rightarrow \overline{NE(Y)}$. Since $\phi$ is crepant, by Proposition~\ref{prop.section} we deduce that $\overline{NE(Y)}$ intesects nontrivially both the open half-space $K_Y < 0$ and the closed half-space $K_Y \ge 0$. Let us denote the unique $K_Y$-negative ray of $ \overline{NE(Y)}$ by $\gamma$. We claim that $\gamma = \phi_*[C_\pi]$. 
	
	Reasoning by contradiction, suppose that $\gamma$ is distinct from $\phi_*[C_\pi]$, so the latter is in the interior of the cone. Then the preimage of $\gamma$ via $\phi_*$ contains a face (of dimension 1 or 2) of the Mori cone, that by crepancy of $\phi$ is $K_{\p(TS)}$-negative. However $[C_\pi] \notin \phi_*^{-1}(\gamma)$, so this is a contradiction with the existence of the unique Mori ray for $\p(TS)$. We conclude that $\phi_*[C_\pi] = \gamma$, so it gives a Mori contraction. It is necessarily of fiber type, as so is $[C_\pi]$ in $\p(TS)$. We denote it by $\pi'$. It is clear that $p \circ \pi = \pi' \circ \phi$, as both compositions are elementary contractions that differ only in the order.  \par
	$S$ is a locally trivial $\p^1$-bundle over $B$, so we may consider $\p(TS)$ as a locally trivial bundle over $B$ with fibers $\mathbb{F}_2$. Observe that $\sigma(\ell) \subset \mathbb{F}_2$ is precisely the minimal section of $\mathbb{F}_2$. If not, then the class of the minimal section of $\mathbb{F}_2$ would be $\sigma_*[\ell] - m \cdot [C_\pi]$ for $m > 0$, but the existence of a curve in such class contradicts the nefness of $D_{a}$ for any $a > 0$. Consequently, the morphism $\phi$ restricted to any fiber of $\p(TS) \rightarrow B$ is the contraction onto a cone $\mathcal{C}_2$, so we conclude that $Y$ is a $\mathcal{C}_2$-bundle over $B$. \par 
	To verify that the resulting variety $Y$ is a contact variety, we need to check that there exists a globally defined contact line bundle $L$ and that the singularities are rational. For the first claim we can simply use \cite[Thm~3.7 (4)]{KM98} to claim that $Y$ admits a line bundle $L$ such that $\phi^* L = \mathcal{O}_{\p(TS)}(1)$ (note that the cited theorem is stated for Mori contractions, but the part that we use does not require contraction to be Mori, only to be given by the linear system of some semiample divisor). \par
	Concerning the class of singularities, consider the situation fiberwise. Over every $b \in B$ we have $\phi_{|\pi^{-1}(b)} \colon \mathbb{F}_2 \rightarrow \mathcal{C}_2$ and this map is a resolution of singularities of a surface $\mathcal{C}_2$. The exceptional divisor $E$ of $\phi_{|\pi^{-1}(b)}$ is a rational curve, so in particular $H^1(E, \mathcal{O}_E) = 0$. By \cite[Cor.~4.9]{Re97} it implies that $R^1 (\phi_{|\mathcal{S}})_* \mathcal{O}_{\mathcal{S}} = 0$, so the singularity is rational and we are done by the local triviality of $Y$. Consequently, $Y$ is a projective singular contact threefold and $\phi$ is its crepant resolution of singularities, so using the notation from the formulation of the theorem, we have $Y=X$ and $\phi = f$.
\end{proof}
We have therefore settled the situation in dimension 3. Unfortunately, to give classification in higher dimensions, one needs more refined arguments: in higher dimensions a resolution of singularities may not produce a projective contact manifold. In the particular case when $\dim (X) = 5$ the terminalization produces a contact variety whose singular locus consists of disjoint, smooth curves by Theorem~\ref{thm.contact.stratification} and resolving them destroys the contact structure as illustrated by Example~\ref{exa.p5}. \par 
Interestingly enough, in dimension 3 smooth examples are more frequent than singular ones: every smooth projective surface $\Sigma$ produces a projective contact manifold of the form~$\p(T\Sigma)$, while as we have just seen, singular contact threefolds essentially correspond to ruled surfaces. On the other hand, just as in the smooth case, the only Fano example comes from the projectivized orbit closure. Nevertheless, the author does not dare to hypothesize that in higher dimensions such singular analogue of LeBrun-Salamon conjecture holds. \par
Finally, as a byproduct of our reasoning, we are able to construct a variety whose singularities are contact in the sense of \cite{CF01}, but that does not satisfy Definition~\ref{defn.contact}:
\begin{exa}\label{exa.CP}[Threefold \textit{\`a la} Campana-Flenner]
	This time let $S$ be the trivial ruled surface over an elliptic curve $E$ given by the rank 2 vector bundle $\mathcal{O}_E \oplus \mathcal{O}_E$, that in particular is semistable and has invariant $e =0$. Any section of the Mori cone of $\p(TS)$ is triangular with two rays lying on the hyperplane $K_{\p(TS)} = 0$. As we have just discussed, the contraction of the ray $\sigma_*[\ell]$ (recall that $\sigma$ by Propostion~\ref{prop.section} corresponds to the surjection $TS \rightarrow TE$) gives us a singular contact threefold. Instead of that, let us contract the other crepant ray, i.e.~given by $\sigma_*[E_0]$, where $E_0 = E \times \{p\}$ for any $p \in \p^1 = \ell$. By a reasoning analogous to the one conducted in this section one can show that the divisor $ D \coloneqq \pi^* E_0 + \pi^* \ell - K_{\p(TS)}$ is big, nef and semiample, so the linear system of some positive multiple of it gives a crepant and birational morphism that will be denoted by $h$. This morphism contracts $\sigma_*[E_0]$, as it is the only class intersected trivially by $D$. \par 
	We can construct (via the pushforward) the contact structure on the resulting variety $Y$ from the one on $\p(TS)$, however there is one subtlety. Namely, for a contact form to be defined around a singular point $x$, we need the contact bundle to be trivial on some neighbourhood of $x$. In our case, we demand $\mathcal{O}_{\p(TS)}(1)$ (the contact bundle on the resolution) to be trivial on $\sigma(E_0)$. By the projection formula and the definition of $\mathcal{O}(1)$, this is equivalent to the triviality of pullback of $TE$ to $E_0$, but it clearly holds as $E$ is elliptic. Consequently, $Y$ has contact singularities in the sense of Campana and Flenner \cite{CF01}. \par 
	However, this variety does not satisfy Definition~\ref{defn.contact}, as the singularities are not rational. It can be checked directly: $Y$ has a resolution of singularities $h \colon \p(TS) \rightarrow Y$. If we restrict ourselves to $\mathcal{S} \coloneqq \p(TS_{|E_0})$, then $h$ becomes a contraction from a ruled surface $\mathcal{S}$ over $E$ onto a cone $\mathcal{C}$ (over $E$). In particular, the exceptional divisor is isomorphic to $E$, so it has $h^1(E, \mathcal{O}_E) =1$. The hyperplane exact sequence for the exceptional divisor in $\mathcal{S}$ gives a surjective homomorphism $R^1 (h_*)_{|\mathcal{S}} \mathcal{O}_{\mathcal{S}} \rightarrow H^1(E, \mathcal{O}_E)$, as shown in \cite[Section~4.8]{Re97}, so in particular $R^1 h_* \mathcal{O}_{\p(TS)} \ne 0$ and the singularities are not rational. \par 
	Finally, let us note another peculiarity of this example. In the symplectic case Namikawa showed that the rationality of singularities is equivalent to the existence of extension of the symplectic form to the resolution \cite[Thm~6]{Nam01}. For the contact example that we have just constructed, the situation is different, as the contact form clearly extends to the resolution, even though the singularities are not rational (they are strictly log canonical). This difference can be justified by the fact that 1-forms on log canonical varieties (pairs) have better extension properties than forms of higher order, see \cite[Thm~1.4]{GK13} and compare with \cite[Thm~1.5]{GKKP11}.
\end{exa}
	\bibliographystyle{alpha}
	\bibliography{biblio}
\end{document}